\documentclass[11pt]{article}
\usepackage{latexsym,amssymb,amsmath,amsthm,enumerate,float,geometry,cite,tikz,bbm,algorithm2e}
\newtheorem{theorem}{Theorem}
\newtheorem{lemma}[theorem]{Lemma}

\newtheorem{conjecture}[theorem]{Conjecture}

\newtheorem{claim}{Claim}

\usepackage{lineno}
\usepackage{setspace}

\allowdisplaybreaks

\begin{document}

\title{Almost color-balanced perfect matchings \\ in color-balanced complete graphs}
\author{Johannes Pardey \and Dieter Rautenbach}

\maketitle
\vspace{-10mm}
\begin{center}
{\small Institute of Optimization and Operations Research, Ulm University,\\ 
Ulm, Germany, \texttt{$\{$johannes.pardey,dieter.rautenbach$\}$@uni-ulm.de}}
\end{center}

\begin{abstract}
For a graph $G$ 
and a not necessarily proper $k$-edge coloring $c:E(G)\to \{ 1,\ldots,k\}$,
let $m_i(G)$ be the number of edges of $G$ of color $i$,
and call $G$ {\it color-balanced} if $m_i(G)=m_j(G)$ for every two colors $i$ and $j$.
Several famous open problems relate to this notion;
Ryser's conjecture on transversals in latin squares, for instance,
is equivalent to the statement 
that every properly $n$-edge colored complete bipartite graph $K_{n,n}$
has a color-balanced perfect matching.
We contribute some results on the question posed by Kittipassorn and Sinsap 
(arXiv:2011.00862v1) 
whether every $k$-edge colored color-balanced complete graph $K_{2kn}$ 
has a color-balanced perfect matching $M$.
For a perfect matching $M$ of $K_{2kn}$, 
a natural measure for the total deviation of $M$ 
from being color-balanced is  
$f(M)=\sum\limits_{i=1}^k|m_i(M)-n|$.
While not every color-balanced complete graph $K_{2kn}$ 
has a color-balanced perfect matching $M$,
that is, a perfect matching with $f(M)=0$,
we prove the existence of a perfect matching $M$
with $f(M)=O\left(k\sqrt{kn\ln(k)}\right)$ for general $k$
and $f(M)\leq 2$ for $k=3$;
the case $k=2$ has already been studied earlier.
An attractive feature of the problem is that it naturally invites the combination
of a combinatorial approach based on counting and local exchange arguments
with probabilistic and geometric arguments.\\[3mm]
{\bf Keywords:} Zero-sum Ramsey theory; rainbow matching
\end{abstract}

\pagebreak

\section{Introduction}

As a special question from the area of zero-sum Ramsey theory \cite{ca,gage},
Caro, Hansberg, Lauri, and Zarb \cite{cahalaza} asked whether every $2$-edge-colored complete graph of order $4n$ 
with equally many edges of both colors always has a perfect matching that also contains equally many edges of both colors.
This question was answered independently and affirmatively 
by Ehard, Mohr, and Rautenbach\cite{ehmora} and 
by Kittipassorn and Sinsap \cite{kisi}.
The motivation for the present paper is an interesting problem formulated by Kittipassorn and Sinsap 
concerning possible generalizations of this result to settings with more than two colors.
In particular, they \cite{kisi} asked whether, 
for every $k$-edge-coloring of $K_{2kn}$
with equally many edges of each color,
there is a perfect matching $M$
that also contains equally many edges of each color;
in other words, whether there is a {\it color-balanced} perfect matching
for every {\it color-balanced} edge-coloring of $K_{2kn}$.
Note that the considered edge-colorings are not required to be proper,
and that an obvious necessary condition for the existence of $M$
is that the order of the complete graph is a multiple of $2k$.
While the example in Figure \ref{fig1} 
shows that there is not always such a matching,
we believe that there are always perfect matchings 
that are very close to being color-balanced. 

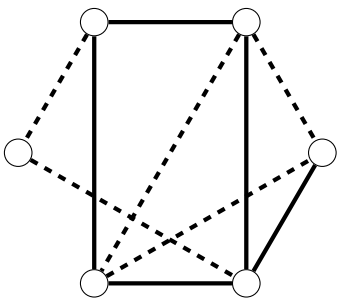
\begin{figure}[h]
    \begin{center}
      
    \begin{tikzpicture}
    \foreach \a in {1,...,6}{
				\node[draw, circle] (u\a)at ({\a*360/6}:2){};
			}
			\draw[ultra thick] (u6) -- (u5) -- (u1) -- (u2) -- (u4) -- (u5);
			
			\draw[ultra thick, dashed] (u2) -- (u3);
			\draw[ultra thick, dashed] (u1) -- (u4);
			\draw[ultra thick, dashed] (u4) -- (u6);
			\draw[ultra thick, dashed] (u3) -- (u5);
			\draw[ultra thick, dashed] (u1) -- (u6);

    \end{tikzpicture}
\caption{For the edge-coloring of the $15$ edges of $K_6$ with three colors 
such that 
the $5$ edges indicated by bold lines form one color class,
the $5$ edges indicated by dashed lines form a second color class,
and the missing $5$ edges form the third color class,
there is no perfect matching containing one edge from each color class.}
\label{fig1}  
\end{center}
\end{figure}
We propose the following conjecture.

\begin{conjecture}\label{conjecture1}
If $n$ and $k$ are positive integers, and $c:E(K_{2kn})\to [k]$ is such that,
for every $i$ in $[k]$,
there are equally many edges $e$ with $c(e)=i$,
then there is a perfect matching $M$ of $K_{2kn}$ with 
$$f(M)\leq O(k^2),$$ 
where
$f(M)=\sum\limits_{i=1}^k\Big||c^{-1}(i)\cap M|-n\Big|.$
\end{conjecture}
Two main results in the present paper are the following,
where $f(M)$ is as in Conjecture \ref{conjecture1}.

\begin{theorem}\label{theoremg0}
If $n$ and $k$ are positive integers, and $c:E(K_{2kn})\to [k]$ is such that,
for every $i$ in $[k]$,
there are equally many edges $e$ with $c(e)=i$,
then there is a perfect matching $M$ of $K_{2kn}$ with 
$$f(M)\leq 3k\sqrt{kn\ln(2k)}.$$ 
\end{theorem}

\begin{theorem}\label{theorem1}
If $n$ is a positive integer, and $c:E(K_{6n})\to [3]$ is such that,
for every $i$ in $[3]$,
there are equally many edges $e$ with $c(e)=i$,
then there is a perfect matching $M$ of $K_{6n}$ with 
$f(M)\leq 2.$
\end{theorem}
The term $f(M)$ is a natural measure 
for the total deviation of $M$ 
from being color-balanced.
It coincides with the Manhattan norm $\|v(M)\|_1$
of the following vector associated with a perfect matching $M$ of $K_{2kn}$
$$v(M)=
\begin{pmatrix}
m_1(M)-n\\
m_2(M)-n\\
\vdots\\
m_k(M)-n
\end{pmatrix}\in \mathbb{R}^{k},$$
where $m_i(M)=|c^{-1}(i)\cap M|$ is the number of edges in $M$ with color $i$
for every $i$ in $[k]$.

Let ${\cal M}$ be the set of all perfect matching of $K_{2kn}$.

Under a weaker hypothesis and 
with an approach combining geometric and probabilistic arguments, 
we obtain our third main result.

\begin{theorem}\label{theoremg1}
If $n$ and $k$ are positive integers with $k\geq 4$, 
and $c:E(K_{2kn})\to [k]$ is such that
the origin lies in the convex hull of the points in 
$\{ v(M):M\in {\cal M}\}$, 
then there is a perfect matching $M$ of $K_{2kn}$ with 
$$f(M)\leq 13k^{\frac{11}{4}}n^{\frac{3}{4}}\sqrt{\ln(2k)}.$$ 
\end{theorem}
Before we proceed to the proofs, we briefly mention some related research.
An extreme example of color-balanced matchings are {\it rainbow matchings}
in which every two edges are colored differently. 
The existence of rainbow matchings has received a tremendous attention,
motivated by the famous conjectures
of Ryser \cite{ry}
and 
of Brualdi and Stein \cite{st}.
Ryser's conjecture states that, 
for every proper edge-coloring with $n$ colors  
of the complete bipartite graph $K_{n,n}$ 
where $n$ is odd, 
there is a perfect rainbow matching,
and the Brualdi-Stein conjecture states that, 
in the same setting for even $n$, there is a rainbow matching of size $n-1$.
Note that, trivially, for every proper edge-coloring of $K_{n,n}$ with $n$ colors,  
every color appears equally often.
The best known lower bound on the size of a largest rainbow matching 
in this setting is due to Hatami and Shor \cite{hash}.
Woolbright considered a similar problem for the complete graph 
instead of the complete bipartite graph, and proved that,
for every proper edge-coloring with $2n-1$ colors of $K_{2n}$,
there is a perfect rainbow matching.
Unlike in our setting, the edge-colorings in these conjectures and results are proper.
The existence of rainbow matchings has also been studied 
for non-proper edge-colorings. 
Erd\H{o}s and Spencer \cite{ersp} proved that,
for every (not necessarily proper) edge-coloring of $K_{n,n}$ 
for which every color is used on at most $\frac{n-1}{4e}$ edges,
there is a perfect rainbow matching. 
Just requiring that every available color is used at least once,
Fujita, Kaneko, Schiermeyer, and Suzuki \cite{fukascsu}
studied the existence of large rainbow matchings in edge-colored complete graphs.
Kostochka and Yancey \cite{koya} 
showed the existence of large rainbow matchings
for edge-colored graphs provided that every vertex is incident with edges of many different colors.

The following three sections contain the proofs of our three results.

\section{Proof of Theorem \ref{theoremg0}}

Throughout this section, let $n$ and $k$ be positive integers,
and let $c:E(K_{2kn})\to [k]$ be such that 
for every $i$ in $[k]$,
there are equally many edges $e$ with $c(e)=i$.
Let $V(K_{2kn})=[2kn]$. 
For $i$ in $[k]$ and a (not necessarily perfect) matching $M$ of $K_{2kn}$, 
let $m_i(M)$ be the number of edges $e$ in $M$ with $c(e)=i$.

We consider the following random procedure constructing a perfect matching of $K_{2kn}$:

\begin{algorithm}[H]
\Begin{
$M\leftarrow \emptyset$; 
$R\leftarrow [2kn]$\;
\While{$R\not=\emptyset$}
{
$x\leftarrow \min(R)$\;
$y\leftarrow$ a vertex from $R\setminus \{ x\}$ selected uniformly at random\;
$M\leftarrow M\cup \{ xy\}$; 
$R\leftarrow R\setminus \{ x,y\}$\;
}
\Return $M$\;
}
\caption{\texttt{RPM}}
\end{algorithm}

\bigskip

Let $M$ be a random matching returned by \texttt{RPM}.
Our goal is to use Azuma's inequality \cite{az} (cf. \cite{more}, p.~92)
in order to show that the random variable $m_i(M)$ is strongly concentrated.
Therefore, let $(e_1,\ldots,e_{j-1},e_j)$ and $(e_1,\ldots,e_{j-1},e'_j)$ 
be two possible sequences of the first $j$ edges $xy$ 
selected by \texttt{RPM} for some $j$. 
Let $K$ be the complete graph that arises by removing from $K_{2kn}$
the $2(j-1)$ vertices incident with the edges $e_1,\ldots,e_{j-1}$.
By the selection of $x$ within \texttt{RPM},
we have $e_j=xy$ and $e_j'=xy'$ for distinct vertices $x$, $y$, and $y'$ of $K$.
Let ${\cal M}_j$ be the set of all perfect matchings of $K$ that contain $e_j$,
and let ${\cal M}'_j$ be the set of all perfect matchings of $K$ that contain $e_j'$.
For $M_j$ in ${\cal M}_j$, let the vertex $x'$ be such that $x'y'\in M_j$, and let
$$g(M_j)=(M_j\setminus \{ xy,x'y'\})\cup \{ xy',x'y\}.$$
Clearly, $g(M_j)\in {\cal M}_j'$.
Note that, given the perfect matching $g(M_j)$ 
and knowing the fixed vertices $x$, $y$, and $y'$,
the vertex $x'$ is uniquely determined as the matching partner 
of $y'$ in $g(M_j)$,
that is, the original perfect matching $M_j$ can be reconstructed,
which means that the function $g:{\cal M}_j\to {\cal M}_j'$ is injective.
Since $|{\cal M}_j|=|{\cal M}_j'|$, it follows that $g$ is actually bijective.

Trivially,
$$\big|m_i(M)-m_i(g(M))\big|\leq 2
\mbox{ for every $M$ in ${\cal M}_j$ and every $i$ in $[k]$.}$$
A simple bijective argument implies that,
conditioned on the choice of $e_1,\ldots,e_{j-1},e_j$,
the procedure \texttt{RPM} produces each extension $M_j$ of $\{ e_j\}$
to a perfect matching of the complete graph $K$ equally likely.
For every $i$ in $[k]$, this implies
\begin{eqnarray*}
&&\Big|\mathbb{E}\Big[m_i(M)\mid (e_1,\ldots,e_{j-1},e_j)\Big]-
\mathbb{E}\Big[m_i(M)\mid (e_1,\ldots,e_{j-1},e_j')\Big]\Big|\\
&=& \left|\frac{1}{|{\cal M}_j|}\sum\limits_{M_j\in {\cal M}_j}m_i(M_j)
-\frac{1}{|{\cal M}_j'|}\sum\limits_{M_j'\in {\cal M}_j'}m_i(M_j')\right|\\
&=& \left|\frac{1}{|{\cal M}_j|}
\sum\limits_{M_j\in {\cal M}_j}\Big(m_i(M_j)-m_i(g(M_j))\Big)\right|\\
&\leq & \frac{1}{|{\cal M}_j|}
\sum\limits_{M_j\in {\cal M}_j}\Big|m_i(M_j)-m_i(g(M_j))\Big|\\
&\leq & 2.
\end{eqnarray*}
Since \texttt{RPM} returns each perfect matching of $K_{2kn}$ equally likely,
we have $\mathbb{E}[m_i(M)]=n$ (cf. Lemma \ref{lemmag1} below), and
Azuma's inequality implies
$$\mathbb{P}\big[|m_i(M)-n|>t\big]\leq 2e^{-\frac{t^2}{8kn}}
\mbox{ for every $t>0$ and every $i$ in $[k]$.}$$
For $t=3\sqrt{kn\ln(2k)}$, the right hand side of the previous inequality is less than $\frac{1}{k}$.
Now, the union bound shows that, with positive probability,
$$f(M)=\sum\limits_{i=1}^k|m_i(M)-n|\leq 3k\sqrt{kn\ln(2k)},$$
which completes the proof. $\Box$

\section{Proof of Theorem \ref{theorem1}}
\setcounter{claim}{0}

Throughout this section, let $n$ be a positive integer, 
and let $c:E(K_{6n})\to \{ {\rm red},{\rm green},{\rm blue}\}$ be such 
that there are equally many red, green, and blue edges,
that is, there are exactly $\frac{1}{3}{6n\choose 2}$ edges of each color.
For a perfect matching $M$ of $K_{6n}$, let
\begin{align*}
    r_M &= |\left\{e \in M : c(e) = \text{ red}\right\}|\\
    g_M &= |\left\{e \in M : c(e) = \text{ green}\right\}|\\
    b_M &= |\left\{e \in M : c(e) = \text{ blue}\right\}|,\mbox{ which implies}\\
    f(M) &=|r_M-n|+|g_M-n|+|b_M-n|.
\end{align*}
For two sets $N$ and $N'$ of edges of $K_{6n}$, 
let $E[N,N']$ be the set of edges $uv$ of $K_{6n}$
such that $u$ is incident with an edge in $N$
and $v$ is incident with an edge in $N'$,
and let $E[N]$ denote $E[N,N]$.

Now, let the perfect matching $M$ of $K_{6n}$ be chosen such that $f(M)$ is as small as possible.
In order to complete the proof, we need to show that $f(M)\leq 2$.
While the claims below are formulated in terms of $M$,
for our arguments it is important 
that the stated properties actually hold for all perfect matchings $M'$ of $K_{6n}$ 
with $f(M')=f(M)$.

Let 
\begin{eqnarray*}
M_r &=& \left\{ e \in M : c(e) = \text{ red} \right\},\\
M_g &=& \left\{ e \in M : c(e) = \text{ green} \right\},\\
M_b&=& \left\{ e \in M : c(e) = \text{ blue} \right\},\\
r&=& r_M=|M_r|,\,\,\,\,
g=g_M=|M_g|,\,\,\,\,\mbox{ and }b=b_M=|M_b|.
\end{eqnarray*}
By symmetry, we may assume 
$$r \geq g \geq b.$$
If $uv$ and $xy$ are two edges in $M$,
then we say that the two edges $ux$ and $vy$ are {\it parallel};
note that the two edges $uy$ and $vx$ are also parallel.
If replacing $uv$ and $xy$ within $M$ by the parallel edges $ux$ and $vy$
yields a perfect matching $M'$ with $f(M')<f(M)$, then we obtain a contradiction to the choice of $M$;
in this case, we say that $\{ uv,xy\}\to\{ ux,vy\}$ is a {\it contradicting swap}.

\begin{claim}\label{claim1}
$|n-g| \leq 1$.
\end{claim}
\begin{proof}[Proof of Claim \ref{claim1}.]
Suppose, for a contradiction, that $|n-g| \geq 2$. 

First, we assume that $g\geq n+2$.
Note that $r\geq n+2$ and $b\leq n-4$.
If $uv,xy\in M_r\cup M_g$ and $ux$ is blue, then the swap $\{ uv,xy\}\to\{ ux,vy\}$ is contradicting.
In fact, if $M'$ denotes the perfect matching generated by the considered swap,
then $r_{M'},g_{M'}\geq n$, $r_{M'}+g_{M'}\leq r+g-1$, and $n\geq b_{M'}\geq b+1$,
which implies $f(M')<f(M)$.
Hence, no edge in $E[M_r \cup M_g]$ is blue.
Similarly, if $uv\in M_b$, $xy\in M_r\cup M_g$, and $ux$ as well as $vy$ are blue, then the swap $\{ uv,xy\}\to\{ ux,vy\}$ is contradicting.
This implies that at most half the edges in $E[M_b,M_r \cup M_g]$ are blue;
more precisely, 
from every two parallel edges in $E[M_b,M_r \cup M_g]$
at most one is blue.
These two observations imply that the total number of blue edges is at most
\begin{eqnarray}\label{e1}
\binom{2b}{2} + \frac{1}{2}2b(6n-2b)
\stackrel{b<n}{<} \binom{2n}{2} + \frac{1}{2}2n(6n-2n)
= \frac{1}{3}{6n\choose 2},
\end{eqnarray}
where the inequality uses 
the strict monotonicity of the binomial coefficient as well as 
of the function $b\mapsto 2b(6n-2b)$.
This contradicts the hypothesis that exactly one third of all edges are blue.

Now, we may assume that $g\leq n-2$.
Note that $r\geq n+4$ and $b\leq n-2$.
If $uv,xy\in M_r$ and $ux$ is not red, then the swap $\{ uv,xy\}\to\{ ux,vy\}$ is contradicting.
Hence, all edges in $E[M_r]$ are red.
Furthermore, if $uv\in M_r$, $xy\in M_g\cup M_b$, and $ux$ as well as $vy$ are not red, then the swap $\{ uv,xy\}\to\{ ux,vy\}$ is contradicting.
This implies that at least half the edges in $E[M_r,M_g \cup M_b]$ are red;
more precisely, from every two parallel edges in $E[M_r,M_g \cup M_b]$
at least one is red.
These two observations imply that the total number of red edges is at least
\begin{eqnarray}\label{e2}
\binom{2r}{2} + \frac{1}{2}\cdot 2r(6n-2r) =(6n-1)r
\stackrel{r>n}{>}(6n-1)n=\frac{1}{3}{6n\choose 2}.
\end{eqnarray}
Again, this is a contradiction, which completes the proof.
\end{proof}

\begin{claim}\label{claim2}
If $|n-g|=1$, then $\{|n-r|,|n-b|\}=\{ k,k+1\}$ for some $k\leq 1$.
\end{claim}
\begin{proof}[Proof of Claim \ref{claim2}.]
Since $(n-g)+(n-r)+(n-b)=0$, 
the equality $|n-g|=1$ implies 
$\{|n-r|,|n-b|\}=\{ k,k+1\}$ for some positive integer $k$ at most $n$.
For a contradiction, suppose that $k\geq 2$,
which trivially implies $n\geq 2$.

There are two relevant possibilities:
\begin{itemize}
\item [(P1)] $(r,g,b)=(n+k+1,n-1,n-k)$
\item [(P2)] $(r,g,b)=(n+k,n+1,n-k-1)$
\end{itemize}
While these two possibilities require different arguments,
we actually show that we can always swap between (P1) and (P2).
This allows to exploit the observations made for both possibilities simultanoeusly.

First, we assume that $M$ realizes possibility (P1).
Similarly as in the second part of the proof of Claim \ref{claim1}, it follows that 
all edges in $E[M_r]$ are red
and at least half the edges in $E[M_r,M_g]$ are red.
If at least half the edges in $E[M_r,M_b]$ are red, then we obtain the same contradiction as in (\ref{e2}).
Hence, there are edges $e_r=uv$ in $M_r$ and $e_b=xy$ in $M_b$ such that $ux$ and $vy$ are both not red.
If one of $e_g=ux$ and $e_g'=vy$ is blue, 
then the swap $\{ uv,xy\}\to\{ ux,vy\}$ is contradicting.
Hence, both edges are green.
Note that the swap $\{ uv,xy\}\to\{ ux,vy\}$ 
yields a matching $M'$ with $f(M')=f(M)$ realizing possibility (P2).
Furthermore, a simple swapping argument implies 
that the parallel edge of every blue edge in $E[M_r,M_b]$ is red.

Next, we assume that $M$ realizes possibility (P2).
Similarly as in the first part of the proof of Claim \ref{claim1}, it follows that 
no edge in $E[M_r,M_r\cup M_g]$ is blue, and
at most half the edges in $E[M_b,M_r\cup M_g]$ are blue.
If no edge in $E[M_g]$ is blue, then we obtain the same contradiction as in (\ref{e1}).
Hence, $E[M_g]$ must contain at least one blue edge.
Furthermore, if $uv,xy\in M_g$, $ux$ is blue, and $vy$ is not red, then the swap $\{ uv,xy\}\to\{ ux,vy\}$ is contradicting.
Hence, the parallel edge of every blue edge in $E[M_g]$ is red.
Note that a swap involving a blue and a parallel red edge in $E[M_g]$ yields a matching $M'$ with $f(M')=f(M)$ realizing possibility (P1).

Hence, we may assume that $M$ realizes possibility (P1).
Let $e_r$, $e_b$, $e_g$, and $e_g'$ be as above.
The perfect matching $M'=(M\setminus \{ e_r,e_b\})\cup \{ e_g,e_g'\}$
satisfies $f(M')=f(M)$, and realizes possibility (P2).
Since $M_g\subseteq M_g'$,
and the parallel edge of every blue edge in $E[M_g']$ is red, 
the parallel edge of every blue edge in $E[M_g]$ is red.
Our goal is to lower bound 
the number of blue edges in $E[M_r,M_b]\cup E[M_g]$;
recall that the parallel of each such blue edge is red.
\begin{itemize}
\item There are no blue edges in $E[M_r]$.
\item Since there are no blue edges in $E[M'_r,M'_g]$, $M_g\subseteq M_g'$, and $|M_r\setminus M_r'|=1$,
there are at most $4g=4(n-1)$ blue edges in $E[M_r,M_g]$.
\item Since at most half the edges in $E[M'_g,M'_b]$ are blue,
more precisely, one of every two parallel edges in $E[M'_g,M'_b]$ is not blue, $M_g\subseteq M_g'$, and $|M_b\setminus M_b'|=1$,
there are at most $\frac{1}{2}4(b-1)g+4g=2(n-1)(b+1)$ blue edges in $E[M_g,M_b]$.
\item Clearly, there are at most ${2b\choose 2}$ blue edges in $E[M_b]$.
\end{itemize}
Since there are $\frac{1}{3}{6n\choose 2}$ blue edges altogether, 
there are at least 
\begin{eqnarray*}
\frac{1}{3}{6n\choose 2}
-4(n-1)
-2(n-1)(b+1)
-{2b\choose 2}
&=&6n^2-4n+2-2k^2+k(2n+1)\\
&\stackrel{2\leq k\leq n}{\geq}& 2n^2+8n-8
\end{eqnarray*}
blue edges in $E[M_r,M_b]\cup E[M_g]$.
As each of these blue edges has a red parallel edge, 
there are at least $2n^2+8n-8$ red edges in $E[M_r,M_b]\cup E[M_g]$.
Since there are ${2r\choose 2}$ red edges in $E[M_r]$
and at least half the edges in $E[M_r,M_g]$ are red,
the total number of red edges is at least 
$${2r\choose 2}
+2rg
+2n^2+8n-8
\stackrel{r\geq n+3,g=n-1}{\geq}6n^2+23n+1>\frac{1}{3}{6n\choose 2}.
$$
This contradiction completes the proof.
\end{proof}

\begin{claim}\label{claim3}
If $g=n$, then $|n-r|=|n-b|\leq 2$.
\end{claim}
\begin{proof}[Proof of Claim \ref{claim3}.]
Clearly, $|n-r|=|n-b|=k$ for some positive integer $k$ at most $n$,
that is, $(r,g,b)=(n+k,n,n-k)$.
For a contradiction, we suppose that $k\geq 3$:

If some edge in $E[M_r]$ is blue, 
then the swap involving this edge and its parallel edge is contradicting. 
Hence, no edge in $E[M_r]$ is blue.
If $e,e'\in E[M_r]$ are parallel such that $e$ is green and $e'$ is red,
then applying Claim \ref{claim2} to the matching that arises from $M$
by the swap involving $e$ and $e'$ implies a contradiction.
If $e,e'\in E[M_r]$ are parallel such that $e$ and $e'$ are both green,
then applying 
Claim \ref{claim1} (for $k\geq 4$)
or 
Claim \ref{claim2} (for $k=3$)
to the matching that arises from $M$
by the swap involving $e$ and $e'$ implies a contradiction.
Hence, all edges in $E[M_r]$ are red.

Similarly, invoking Claim \ref{claim2} for suitable matchings,
we obtain that 
no edge in $E[M_r,M_g]$ is blue,
at least half the edges in $E[M_r,M_g]$ are red,
and 
at most half the edges in $E[M_g,M_b]$ are blue.
This implies that the number of blue edges in $E[M_r,M_b]\cup E[M_g]$
is at least
$$
\frac{1}{3}{6n\choose 2}-\frac{1}{2}2b2g-{2b\choose 2}
=2n^2-2k^2+(6n-1)k
\stackrel{3\leq k\leq n}{\geq}
2n^2+18n-21.
$$
Again 
invoking Claim \ref{claim2} for suitable matchings,
we obtain that every blue edge in $E[M_r,M_b]$ has a parallel red edge,
and 
invoking Claim \ref{claim1} or Claim \ref{claim2} for suitable matchings,
we obtain that every blue edge in $E[M_g]$ has a parallel red edge.
This implies that there are at least $2n^2+18n-21$
red edges in $E[M_r,M_b]\cup E[M_g]$.
Now, the total number of red edges is at least
$${2r\choose 2}
+2rg
+2n^2+18n-21
=6n^2+17n-21+2k^2+(6n-1)k
\stackrel{3\leq k\leq n}{\geq}
6n^2+35n-6
>\frac{1}{3}{6n\choose 2}.
$$
This contradiction completes the proof.
\end{proof}
Claims \ref{claim1}, \ref{claim2}, and \ref{claim3} already imply $f(M)\leq 4.$
In order to complete the proof, we suppose, for a contradiction, that $f(M)=4$,
which implies
$$(r,g,b)\in \{ (n+1,n+1,n-2),(n+2,n-1,n-1),(n+2,n,n-2)\}.$$

\begin{claim}\label{claim4}
$(r,g,b)\not=(n+1,n+1,n-2)$.
\end{claim}
\begin{proof}[Proof of Claim \ref{claim4}.]
Suppose, for a contradiction, that $(r,g,b)=(n+1,n+1,n-2)$.

Simple swapping arguments imply
that there are no blue parallel edges in $E[M_b,M_r\cup M_g]$,
which implies that at most half the edges in $E[M_b,M_r\cup M_g]$ are blue.
Similarly, there are no blue edges in $E[M_r,M_g]$,
every blue edge in $E[M_r]$ has a green parallel, and 
every blue edge in $E[M_g]$ has a red parallel.
If no edge in $E[M_r]\cup E[M_g]$ is blue, 
then we obtain the same contradiction as in (\ref{e1}).
Hence, by the symmetry between red and green in this case,
we may assume that 
$uv$ and $xy$ are green edges such that 
$e_b=ux$ is blue and $vy$ is red.
The perfect matching $M'=(M\setminus \{ uv,xy\})\cup \{ e_b,vy\}$ 
satisfies
$f(M')=f(M)$,
$M_r\subseteq M_r'$, and 
$(r_{M'},g_{M'},b_{M'})=(n+2,n-1,n-1)$.
A simple swapping argument implies that all edges in $E[M_r']$ are red,
which implies that all edges in $E[M_r]$ are red.

If there are parallel edges in $E[M_r,M_b]$ that are blue and green,
then the swap involving these two edges yields a perfect matching $M''$ 
satisfying
$f(M'')=f(M)$,
$M_g\subseteq M_g''$, and 
$(r_{M''},g_{M''},b_{M''})=(n,n+2,n-2)$.
Now, a simple swapping argument implies that no edge in $E[M_g'']$ is blue,
which contradicts the existence of $e_b$.
Hence, every blue edge in $E[M_r,M_b]$ has a red parallel.

If there are parallel edges in $E[M_r,M_g]$ that are both green,
then the swap involving these two edges yields a perfect matching $M'''$ 
satisfying
$f(M''')=f(M)$,
$E[M_g]\subseteq E[M_g''']$, and 
$(r_{M'''},g_{M'''},b_{M'''})=(n,n+2,n-2)$.
Again, a simple swapping argument implies that no edge in $E[M_g''']$ is blue,
which contradicts the existence of $e_b$.
Since there are no blue edges in $E[M_r,M_g]$, 
it follows that every green edge in $E[M_r,M_g]$ has a red parallel,
and that at least half the edges in $E[M_r,M_g]$ are red.

The number of blue edges in $E[M_g]\cup E[M_r,M_b]$ is at least
$$
\frac{1}{3}{6n\choose 2}-\frac{1}{2}2b2g-{2b\choose 2}
=2n^2+10n-6.
$$
Since each of these blue edges has a red parallel, 
the total number of red edges is at least
$$
{2r\choose 2}+\frac{1}{2}2r2g+2n^2+10n-6
=6n^2+17n-3
>\frac{1}{3}{6n\choose 2},
$$
which is a contradiction.
\end{proof}

\begin{claim}\label{claim5}
$(r,g,b)\not=(n+2,n-1,n-1)$.
\end{claim}
\begin{proof}[Proof of Claim \ref{claim5}.]
Suppose, for a contradiction, that $(r,g,b)=(n+2,n-1,n-1)$.

A simple swapping argument implies 
that all edges in $E[M_r]$ are red.

Let $e$ and $e'$ be two parallel edges in $E[M_r,M_b]$.
If $e$ and $e'$ are both blue
or one is blue and one is green,
then a simple swapping argument implies a contradiction.
If $e$ and $e'$ are both green,
and $M'$ is the perfect matching obtained from $M$
by the swap involving $e$ and $e'$,
then $f(M')=f(M)$, 
and $(r_{M'},g_{M'},b_{M'})=(n+1,n+1,n-2)$,
in which case Claim \ref{claim4} implies a contradiction.
Hence, at least one of every two parallel edges in $E[M_r,M_b]$ is red, 
which implies that at least half the edges in $E[M_r,M_b]$ are red.
By the symmetry between green and blue in this case,
it also follows that at least half the edges in $E[M_r,M_g]$ are red.

Now, the total number of red edges is at least
$$
{2r\choose 2}+\frac{1}{2}2r(6n-2r)
=6n^2 + 11n - 2
>\frac{1}{3}{6n\choose 2},
$$
which is a contradiction.
\end{proof}
We are now in a position to derive our final contradiction.

The previous claims imply $(r,g,b)=(n+2,n,n-2)$.

A simple swapping argument implies that 
no edge in $E[M_r]$ is blue.
Simple swapping arguments and Claim \ref{claim5} imply
that no edge in $E[M_r,M_g]$ is blue.
Claim \ref{claim4} implies
that at least half the edges in $E[M_r,M_g]$ are red.
Simple swapping arguments and Claim \ref{claim4} imply
that every blue edge in $E[M_r,M_b]$ has a red parallel;
in particular, at most half the edges in $E[M_r,M_b]$ are blue.
Claim \ref{claim5} implies
that at most half the edges in $E[M_g,M_b]$ are blue.
Similarly as before, 
if no edge in $E[M_g]$ is blue, 
then there are not enough blue edges in total.
Hence, we may assume that $e$ is a blue edge in $E[M_g]$.
Let $e'$ be the parallel edge of $e$.
Claim \ref{claim5} implies
that $e'$ is not green.

Suppose, for a contradiction, that $e'$ is blue.
If $M'$ is the perfect matching that arises from $M$
by the swap involving $e$ and $e'$, then $f(M')=f(M)$
and $(r_{M'},g_{M'},b_{M'})=(n+2,n-2,n)$.
Repeating the above arguments for $M'$
with blue replaced by green, 
and using $M_r=M_r'$ and $M_b\subseteq M_b'$,
it follows that 
no edge in $E[M_r]$ is green,
and that at least half the edges in $E[M_r,M_b]$ are red.
This implies the contradiction 
that the total number of red edges is at least
$$
{2r\choose 2}+\frac{1}{2}2r(6n-2r)
=6n^2 + 11n - 2
>\frac{1}{3}{6n\choose 2}.
$$
Hence $e'$ is red, or more generally,
every blue edge in $E[M_g]$ has a red parallel.

If $E[M_r]$ contains a green edge,
then either the parallel edge is red, in which case Claim \ref{claim4} implies a contradiction,
or the parallel edge is green, 
in which case the existence of the blue edge $e$ in $E[M_g]$ 
implies a contradiction.
Hence, all edges in $E[M_r]$ are red.

The number of blue edges in $E[M_r,M_b]\cup E[M_g]$ is at least
$$
\frac{1}{3}{6n\choose 2}-\frac{1}{2}2b2g-{2b\choose 2}
=2n^2+12n-10.
$$
Since each of these blue edges has a red parallel, 
the total number of red edges is at least
$$
{2r\choose 2}+\frac{1}{2}2r2g+2n^2+12n-10
=6n^2+23n-4
>\frac{1}{3}{6n\choose 2}.
$$
This final contradiction completes the proof.
$\Box$

\section{Proof of Theorem \ref{theoremg1}}

Throughout this section, let $n$ and $k$ be positive integers with $k\geq 4$, 
and let $c:E(K_{2kn})\to [k]$ be such that 
the origin lies in the convex hull of the points in 
$\{ v(M):M\in {\cal M}\}$, 
where $v(M)$ and ${\cal M}$ are defined as before Theorem \ref{theoremg1}.

Before we prove Theorem \ref{theoremg1},
we explain why its hypothesis is weaker than 
the hypothesis of Theorem \ref{theoremg0}.
In fact, this follows quickly by exploiting symmetries of complete graphs.

\begin{lemma}\label{lemmag1}
If $c$ is such that,
for every $i$ in $[k]$,
there are equally many edges $e$ with $c(e)=i$, then 
$$\sum_{M\in {\cal M}}v(M)=0,$$
in particular, the origin lies in the convex hull of the points in 
$\{ v(M):M\in {\cal M}\}$.
\end{lemma}
\begin{proof}[Proof of Lemma \ref{lemmag1}]
For a matching $M$ in ${\cal M}$ and an edge $e$ of $K_{2kn}$,
let $\mathbbm{1}_{e,M}$ be the indicator variable that is $1$ 
exactly if $e$ belongs to $M$.
Let $p=|{\cal M}|$.
By symmetry, the value 
$q=\sum\limits_{M\in {\cal M}}\mathbbm{1}_{e,M}$
is the same for every edge $e$ of $K_{2kn}$,
that is, every edge of $K_{2kn}$ is contained in exactly $q$ perfect matchings.
Furthermore, since every perfect matching contains exactly one of the $2kn-1$
edges incident with a fixed vertex, we obtain $p=(2kn-1)q$.
Now, for $i\in [k]$, the $i$th coordinate of 
$\sum\limits_{M\in {\cal M}}v(M)$ equals
\begin{eqnarray*}
\sum\limits_{M\in {\cal M}}\left(\left(\sum\limits_{e\in c^{-1}(i)}\mathbbm{1}_{e,M}\right)-n\right)&=&
\left(\sum\limits_{e\in c^{-1}(i)}\underbrace{\sum\limits_{M\in {\cal M}}\mathbbm{1}_{e,M}}_{=q}\right)-pn\\
&=& \left(\sum\limits_{e\in c^{-1}(i)}\frac{p}{2kn-1}\right)-pn\\
&=& \frac{1}{k}{2kn\choose 2}\frac{p}{2kn-1}-pn\\
&=&0,
\end{eqnarray*}
which completes the proof.
\end{proof}
Now, we collect the ingredients of the proof of Theorem \ref{theoremg1}.

\begin{lemma}\label{lemmag2}
If $M_1$ and $M_2$ are two perfect matchings of $K_{2kn}$,
then there is a perfect matching $M_2'$ of $K_{2kn}$,
and a collection ${\cal C}$ of disjoint subgraphs of $K_{2kn}$ such that 
\begin{enumerate}[(i)]
\item $M_2'=M_1\Delta E\left(\bigcup\limits_{C\in {\cal C}}C\right)$,
\item $\|v(M_2)-v(M_2')\|_1\leq 2\sqrt{kn}$, 
\item $|{\cal C}|\leq 3\sqrt{kn}$, and
\item every $C$ in ${\cal C}$ is the union of disjoint $M_1$-alternating cycles,
and satisfies $m(C)\leq 5\sqrt{kn}$.
\end{enumerate}
\end{lemma}
\begin{proof}[Proof of Lemma \ref{lemmag2}]
Let ${\cal C}_0$ be the collection of all $M_1$-$M_2$-alternating cycles 
formed by the symmetric difference of $M_1$ and $M_2$,
that is, all cycles of the graph $\left(V(K_{2kn}),M_1\Delta M_2\right)$.

Let ${\cal C}_1$ arise from ${\cal C}_0$
be replacing every cycle $C$ in ${\cal C}_0$ of length more than $5\sqrt{kn}$
by several disjoint cycles $C_1,\ldots,C_r$ with 
$V(C)=V(C_1)\cup\ldots\cup V(C_r)$ as follows:
Let $u_1v_1,\ldots,u_rv_r$ be edges from $M_2$ on $C$ such that 
all components of $C-\{ u_1v_1,\ldots,u_rv_r\}$ contain
at least $\left\lceil\sqrt{kn}\right\rceil$ and less than 
$2\left\lceil\sqrt{kn}\right\rceil$ edges from $M_1$.
Note that 
$$r\leq \left\lfloor\frac{m(C)}{2\left\lceil\sqrt{kn}\right\rceil}\right\rfloor,$$
and that the components of $C-\{ u_1v_1,\ldots,u_rv_r\}$
are paths $P_1,\ldots,P_r$.
We may assume that 
$$u_1,v_1,P_1,u_2,v_2,P_2,\ldots,u_r,v_r,P_r$$
is the cyclic order in which the incident vertices and paths appear along $C$,
that is, $P_i$ is a path between $v_i$ and $u_{i+1}$.
Now replace $C$ within ${\cal C}_0$ by the cycles
$$P_1+v_1u_2,
P_2+v_2u_3,
\ldots,
P_r+v_ru_1.$$
Note that each of these cycles is $M_1$-alternating,
and contains less than $2\left\lceil\sqrt{kn}\right\rceil$ edges from $M_1$;
in particular, it has length at most 
$2\left(2\left\lceil\sqrt{kn}\right\rceil-1\right)\leq 5\sqrt{kn}$.

Let 
$$M_2'=M_1\Delta E\left(\bigcup\limits_{C\in {\cal C}_1}C\right).$$
Note that $M_2=M_1\Delta E\left(\bigcup\limits_{C\in {\cal C}_0}C\right).$

By construction, we obtain, 
\begin{eqnarray*}
\|v(M_2)-v(M_2')\|_1
&\leq &|M_2\Delta M_2'| \\
&\leq &\sum\limits_{C\in {\cal C}_0:m(C)>5\sqrt{kn}}2\left\lfloor\frac{m(C)}{2\left\lceil\sqrt{kn}\right\rceil}\right\rfloor\\
&\leq &\sum\limits_{C\in {\cal C}_0}\frac{2m(C)}{2\sqrt{kn}}\\
&\leq &\frac{4kn}{2\sqrt{kn}}\\
&\leq &2\sqrt{kn}
\end{eqnarray*}
In order to obtain ${\cal C}$, 
we group the cycles in ${\cal C}_1$ that have length less than $\sqrt{kn}$;
call these cycles {\it short}.
In fact, there are disjoint subgraphs $H_1,\ldots,H_s$ of $K_{2kn}$
such that 
\begin{itemize}
\item 
each $H_i$ is the disjoint union of short cycles from ${\cal C}_1$,
\item $m(H_i)\leq 4\sqrt{kn}$ for every $i$ in $[s]$,
\item $m(H_i)\geq \sqrt{kn}$ for all but at most one index $i$ in $[s]$, and
\item $H_1\cup \ldots \cup H_s$ 
equals the union of all short cycles in ${\cal C}_1$.
\end{itemize}
Let ${\cal C}$ arise from ${\cal C}_1$ by replacing all short cycles
with the subgraphs $H_1,\ldots,H_s$.
Note
that $M_2'=M_1\Delta E\left(\bigcup\limits_{C\in {\cal C}}C\right)$, 
every subgraph $C$ in ${\cal C}$ is the disjoint union of $M_1$-alternating cycles,
and satisfies $m(C)\leq 5\sqrt{kn}$.
The lower bound on $m(H_i)$ for all but one index $i$ implies that 
$$|{\cal C}|\leq \frac{2nk}{\sqrt{kn}}+1\leq 3\sqrt{kn},$$
which completes the proof.
\end{proof}

\begin{lemma}\label{lemmag3}
If $M_1$, $M_2'$, and ${\cal C}$ are as in Lemma \ref{lemmag2} and $p\in [0,1]$,
then there is a subset ${\cal C}'$ of ${\cal C}$ such that 
$$M=M_1\Delta E\left(\bigcup\limits_{C\in {\cal C}'}C\right)$$
is a perfect matching of $K_{2kn}$ 
with 
$$\Big\| v(M)-\Big(pv(M_1)+(1-p)v(M_2')\Big)\Big\|_1\leq 13k^{\frac{7}{4}}n^{\frac{3}{4}}\sqrt{\ln(2k)}.$$
\end{lemma}
\begin{proof}[Proof of Lemma \ref{lemmag3}]
For $C$ in ${\cal C}$, let
$$x(C)=v(M_1\Delta E(C))-v(M_1).$$
By the structure of ${\cal C}$, we have
\begin{itemize}
\item $v(M_2')=v(M_1)+\sum\limits_{C\in {\cal C}}x(C)\mbox{, and}$
\item $\| x(C)\|_1\leq m(C)\leq 5\sqrt{kn}$ for every $C$ in ${\cal C}$.
\end{itemize}
Now, let ${\cal C}'$ be a random subset of ${\cal C}$
containing each element of ${\cal C}$ independently at random with probability $p$.
Clearly, by the structure of ${\cal C}$,
$$M=M_1\Delta E\left(\bigcup\limits_{C\in {\cal C}'}C\right)$$
is a (random) perfect matching of $K_{2kn}$, and
$$v(M)=v(M_1)+\sum\limits_{C\in {\cal C}'}x(C).$$
For every $i$ in $[k]$, the expected value of the $i$th component $v(M)_i$ of $v(M)$
equals exactly the $i$th component $\Big(pv(M_1)+(1-p)v(M_2')\Big)_i$
of $pv(M_1)+(1-p)v(M_2')$.
Furthermore, by the Simple Concentration Bound from \cite{more},
$$
\mathbb{P}\left[
\left|v(M)_i-\Big(pv(M_1)+(1-p)v(M_2')\Big)_i\right|>t
\right]
\leq 2e^{-\frac{t^2}{2\cdot \left(5\sqrt{kn}\right)^2\cdot \left(3\sqrt{kn}\right)}}.
$$
Choosing $t=13(kn)^{\frac{3}{4}}\sqrt{\ln(2k)}$,
the right hand side is less than $\frac{1}{k}$.
Now, the union bound implies that, with positive probability,
$$\left|v(M)_i-\Big(pv(M_1)+(1-p)v(M_2')\Big)_i\right|
\leq 13(kn)^{\frac{3}{4}}\sqrt{\ln(2k)}$$
for each of the $k$ choices of $i$, which implies
$$\left\|v(M)-\Big(pv(M_1)+(1-p)v(M_2')\Big)\right\|_1
\leq 13k^{\frac{7}{4}}n^{\frac{3}{4}}\sqrt{\ln(2k)},$$
completing the proof.
\end{proof}
We are now in a position to complete the proof of Theorem \ref{theoremg1}.

By the theorem of Carath\'{e}odory \cite{cara},
there are $k+1$ (not necessarily distinct) perfect matchings 
$M_1,\ldots,M_{k+1}$ of $K_{2kn}$ 
as well as positive real coefficients
$p_1,\ldots,p_{k+1}$ with $p_1+\cdots+p_{k+1}=1$ such that 
$$p_1v(M_1)+p_2v(M_2)+\cdots+p_{k+1}v(M_{k+1})$$
is the all-zero vector $0$.

We now apply Lemmas \ref{lemmag2} and \ref{lemmag3} exactly $k$ times
in order to reduce the number of vectors 
in the convex combination one by one down to just a single vector $v$
with $v=v(M)$ for some perfect matching $M$ of $K_{2kn}$.
The statement of Theorem \ref{theoremg1} is obtained 
by controling the total deviation/error caused 
by the applications of Lemma \ref{lemmag3}.
Therefore, suppose that 
$$x_i=(p_1+\cdots+p_i)v(M_{[i]})+p_{i+1}v(M_{i+1})+\cdots+p_{k+1}v(M_{k+1})$$
for some $i\in [k]$, where 
$M_{[i]}$ is some perfect matching of $K_{2kn}$ 
that coincides with $M_1$ for $i=1$,
that is, $x_1$ is the all-zero vector.
Let $p=\frac{p_1+\cdots+p_i}{p_1+\cdots+p_i+p_{i+1}}$.
Applying Lemma \ref{lemmag2} to the perfect matchings
$M_{[i]}$
and
$M_{i+1}$
(as $M_1$ and $M_2$ in Lemma \ref{lemmag2})
yields a perfect matching $M_{i+1}'$ 
(corresponding to $M_2'$ in Lemma \ref{lemmag2}).
Now, applying Lemma \ref{lemmag3} 
to the perfect matchings
$M_{[i]}$
and
$M_{i+1}'$
(as $M_1$ and $M_2'$ in Lemma \ref{lemmag3})
and $p$ as above,
yields a perfect matching
$M_{[i+1]}$
(corresponding to $M$ in Lemma \ref{lemmag3})
with 
\begin{eqnarray*}
&&\Big\| (p_1+\cdots+p_{i+1})v(M_{[i+1]})-
\Big(
(p_1+\cdots+p_i)v(M_{[i]})+p_{i+1}v(M_{i+1})
\Big)\Big\|_1 \\
& = & 
\underbrace{(p_1+\cdots+p_{i+1})}_{\leq 1}\Big\| v(M_{[i+1]})-\Big(pv(M_{[i]})+(1-p)v(M_{i+1}')\Big)\Big\|_1\\
&\leq & 13k^{\frac{7}{4}}n^{\frac{3}{4}}\sqrt{\ln(2k)}.
\end{eqnarray*}
Now, if $x_{i+1}$ is defined as
$$x_{i+1}=(p_1+\cdots+p_i+p_{i+1})v(M_{[i+1]})+p_{i+2}v(M_{i+2})+\cdots+p_{k+1}v(M_{k+1}),$$
then
$$\|x_{i+1}-x_i\|_1\leq 13k^{\frac{7}{4}}n^{\frac{3}{4}}\sqrt{\ln(2k)}.$$
This yields a sequence of vectors $x_1,\ldots,x_{k+1}$ in $\mathbb{R}^k$, where 
\begin{itemize}
\item $\|x_1\|_1=0$,
\item $\|x_{i+1}-x_i\|_1\leq 13k^{\frac{7}{4}}n^{\frac{3}{4}}\sqrt{\ln(2k)}$ for every $i$ in $[k]$, and
\item $x_{k+1}=\underbrace{(p_1+\cdots+p_{k+1})}_{=1}v(M_{[k+1]})$ 
for some perfect matching $M_{[k+1]}$ of $K_{2kn}$.
\end{itemize}
Hence,
$$f(M_{[k+1]})=\|v(M_{[k+1]})\|_1\leq \|x_1\|_1+\|x_2-x_1\|_1+\cdots+\|x_{k+1}-x_k\|_1
\leq 13k^{\frac{11}{4}}n^{\frac{3}{4}}\sqrt{\ln(2k)},$$
which completes the proof of Theorem \ref{theoremg1}. $\Box$


\begin{thebibliography}{}
\bibitem{az} K. Azuma, Weighted sums of certain dependent random variables, Tohoku Mathematical Journal 19 (1967) 357-367.
\bibitem{cara} C. Carath\'{e}odory, \"{U}ber den Variabilit\"{a}tsbereich der Koeffizienten von Potenzreihen, die gegebene Werte nicht annehmen, Mathematische Annalen 64 (1907) 95-115.
\bibitem{ca} Y. Caro, Zero-sum problems - a survey, Discrete Mathematics 152 (1996) 93-113.
\bibitem{cahalaza} Y. Caro, A. Hansberg, J. Lauri, and C. Zarb, On zero-sum spanning trees and zero-sum connectivity, arXiv 2007.08240v1. 
\bibitem{ehmora} S. Ehard, E. Mohr, and D. Rautenbach, Low weight perfect matchings, The Electronic Journal of Combinatorics 27 (2020) P4.49.
\bibitem{ersp} P. Erd\H{o}s and J. Spencer, Lopsided Lov\'{a}sz local lemma and latin transversals, Discrete Applied Mathematics 30 (1991) 151-154.
\bibitem{fukascsu} S. Fujita, A. Kaneko, I. Schiermeyer, and K. Suzuki, A rainbow $k$-matching in the complete graph with $r$ colors, The Electronic Journal of Combinatorics 16 (2009) \# R51.
\bibitem{gage} W. Gao and A. Geroldinger, Zero-sum problems in finite abelian groups: A survey, Expositiones Mathematicae 24 (2006) 337-369.
\bibitem{hash} P. Hatami and P.W. Shor, A lower bound for the length of a partial transversal in a Latin square, Journal of Combinatorial Theory, Series A 115 (2008) 1103-1113.
\bibitem{kisi} T. Kittipassorn and P. Sinsap, On the existence of zero-sum perfect matchings of complete graphs, arXiv:2011.00862v1.
\bibitem{koya} A. Kostochka and M. Yancey, Large rainbow matchings in edge-coloured graphs, Combinatorics, Probability and Computing 21 (2012) 255-263.
\bibitem{mopara} E. Mohr, J. Pardey, and D. Rautenbach, Zero-sum copies of spanning forests in zero-sum complete graphs, arXiv:2101.11233v1.
\bibitem{more} M. Molloy and B. Reed, Graph Colouring and the Probabilistic Method, Springer-Verlag Berlin Heidelberg 2002, DOI 10.1007/978-3-642-04016-0.
\bibitem{ry} H.J. Ryser, Neuere Probleme der Kombinatorik, in ``Vortr\"{a}ge \"{u}ber Kombinatorik Oberwolfach'', Mathematisches Forschungsinstitut Oberwolfach, July 1967, 24-29.
\bibitem{st} S.K. Stein, Transversals of Latin squares and their generalizations, Pacific Journal of Mathematics 59 (1975) 567-575.
\bibitem{wofu} D.E. Woolbright and H.-L. Fu, On the existence of rainbows in $1$-factorizations of $K_{2n}$, Journal of Combinatorial Designs 6 (1998) 1-20.
\end{thebibliography}
\end{document}